\documentclass{article}

\usepackage[latin1]{inputenc}
\usepackage{amsfonts, amsmath, amsthm, amssymb}
\usepackage{color, graphicx, caption, booktabs,minibox,url}
\usepackage{lscape,fullpage}

\newtheorem{theorem}{Theorem}

\newtheorem{lemma}[theorem]{Lemma}

\newcommand{\ds}{\displaystyle}

\def\cS{\mathcal{S}}
\def\cA{\mathcal{A}}\def\cB{\mathcal{B}}
\def\cR{\mathcal{R}}
\def\cM{\mathcal{M}}
\def\cN{\mathcal{N}}\def\cN{\mathcal{N}}
\def\cD{\mathcal{D}}\def\cP{\mathcal{P}}
\def\cL{\mathcal{L}}\def\cH{\mathcal{H}}
\def\cF{\mathcal{F}}\def\cT{\mathcal{T}}

\title{Enumeration of labelled 4-regular planar graphs II: asymptotics}
\date{}
\author{
    Marc Noy
    \thanks{Departament de Matemàtiques and Institut de Matemàtiques de la UPC (IMTech), Universitat Politècnica de Catalunya \& Centre de Recerca Matemàtica, Barcelona, Spain. 
    E-mail: {\tt marc.noy@upc.edu}.}
\and 	
    Cl\'ement Requil\'e	
    \thanks{Departament de Matemàtiques and Institut de Matemàtiques de la Universitat Politècnica de Catalunya (IMTech), Barcelona, Spain.
    E-mail: {\tt clement.requile@upc.edu}.}
\and
    Juanjo Ru\'e
    \thanks{Departament de Matemàtiques and Institut de Matemàtiques de la UPC (IMTech), Universitat Politècnica de Catalunya \& Centre de Recerca Matemàtica, Barcelona, Spain. 
    E-mail: {\tt juan.jose.rue@upc.edu}.}
}

\begin{document}

\maketitle

\abstract{
This work is a follow-up of the article [Proc.\ London Math.\ Soc.\ 119(2):358--378, 2019], where the authors solved the problem of counting labelled 4-regular planar graphs.
In this paper, we obtain a precise asymptotic estimate for the number $g_n$ of labelled 4-regular planar graphs on $n$ vertices.
Our estimate is of the form $g_n \sim g\cdot n^{-7/2} \rho^{-n} n!$, where $g>0$ is a constant and $\rho \approx 0.24377$ is the radius of convergence of the generating function $\sum_{n\ge 0}g_n x^n/n!$, and conforms to the universal pattern obtained previously in the enumeration of several classes of planar graphs.
In addition to analytic methods, our solution needs intensive use of computer algebra in order to deal with large systems of multivariate polynomial equations.
We also obtain asymptotic estimates for the number of 2- and 3-connected 4-regular planar graphs, and for the number of 4-regular simple maps, both connected and 2-connected.
}

%
%
\section{Introduction and statement of results}

This paper is a follow-up of \cite{NRR19}, where the authors solved the problem of counting labelled 4-regular planar graphs.
The solution was based on decomposing a 4-regular planar graph along its 3-connected components and finding equations relating the generating functions associated to several classes of planar graphs and maps. 
Our main contribution here is a precise asymptotic estimate for the number of 4-regular planar graphs. 

Starting from the systems of equations in \cite{NRR19}, we determine the dominant singularities of the generating functions and compute the corresponding singular expansions.
Since the functions involved are in all cases algebraic, this can be done through the computation of discriminants (selecting the right factor) and Puiseux expansions (selecting the right branch). 
Then, using singularity analysis \cite{FS09} we show that the number of labelled 4-regular planar graphs on $n$ vertices is asymptotically (understood throughout the paper as $n\to\infty$) equal to
\begin{equation}\label{eq:main-asymp}
	g_n \sim g\cdot n^{-7/2} \gamma^{n} n!,
\end{equation}
where $g>0$ is a constant, $\gamma =\rho^{-1} \approx 4.10228$ and $\rho$ is the radius of convergence of $\sum_{n\ge 0}g_n x^n/n!$.
Analogous results hold for the number of 2- and 3-connected 4-regular planar graphs, and also for simple 4-regular maps.
The estimates conform in all cases to the universal pattern obtained previously in the enumeration of planar graphs and maps \cite{BKLMcD07,BGW02,GN09,NRR19+}.

The asymptotic enumeration of labelled planar graphs and related classes of graphs is currently an active area of research. 
A relevant starting point was the enumeration of 2-connected planar graphs \cite{BGW02}, which opened the way to the full enumeration of all planar graphs \cite{GN09}, and to the enumeration of graphs on surfaces \cite{CFGMN11} and  of several minor-closed classes of graphs \cite{GGNW08,GNR13}.

Cubic (that is, 3-regular) planar graphs were counted in \cite{BKLMcD07} and analyzed further in \cite{NRR19+}.
In all these cases the solution is obtained through generating functions. 
Let $\mathcal{A}$ be a class of labelled graphs closed under isomorphism and let $A(x)=\sum a_n x^n/n!$ be the associated generating function, where $a_n$ is the number of graphs in $\mathcal{A}$ with $n$ vertices. 
The first task is to locate the radius of convergence $\rho$ of $A(x)$ which in most cases is the unique singularity of smallest modulus (an exception are cubic graphs since they have necessarily an even number of vertices, so that $A(x)$ is even and both $\pm \rho$ are singularities). 
This gives the exponential growth $\rho^{-n}n!$ for $a_n$.

The next task is to obtain the subexponential term, which for all the classes described, including 4-regular planar graphs in this paper, is of the form $c\cdot n^{-7/2}$ for some $c>0$. 
The rationale for this pattern comes from the enumeration of \emph{planar maps} (see Section \ref{sec:prelim}). 
For all `natural' classes of planar maps, the subexponential term is $n^{-5/2}$. 
This phenomenon has been explained in a number of ways. 
Probably the most combinatorial explanation is that maps are in bijection with `enriched' trees up to the choice of the root vertex in the tree (see \cite[Chapter 5]{handbook15}), but it can also be explained analytically  \cite{DNY22}.
Since the subexponential term for trees is systematically $n^{-3/2}$, the one for maps becomes $n^{-3/2}/n = n^{- 5/2}$. 
In the enumeration of planar graphs up to date, the starting point has been the subclass of 3-connected graphs in $\mathcal{A}$.
But since by \textit{Whitney's theorem} \cite{W33} a 3-connected planar graph admits a unique embedding in the sphere up to the choice of the orientation, counting 3-connected planar graphs is equivalent to counting 3-connected planar \emph{maps}. 
Typically the generating function of maps are algebraic, leading to singularities of square-root type, hence to subexponential terms of the form $n^{-\alpha}$, where $\alpha$ is a half integer. 
Maps are rooted at an edge, hence the subexponential term in classes of \emph{unrooted} maps is typically $n^{-5/2}/n = n^{-7/2}$. 
The same is true for planar graphs: graphs rooted at a vertex or at an edge give rise to an $n^{-5/2}$ term, while for unrooted graphs it is $n^{-7/2}$, as in our main result \eqref{eq:main-asymp}. 
However, this is not true for series-parallel or outerplanar graphs \cite{BGKN07} and more generally for `subcritical' classes of graphs \cite{DFKKR11}.
The reason is that graphs in these classes are made of a linear number of small blocks and behave like trees, so that the subexponential term for unlabelled graphs is $n^{-5/2}$ instead of $n^{-7/2}$ (see~\cite{GNR13} for a general discussion).

Several open problems remain in the enumeration of planar graphs, most notable the enumeration of \emph{unlabelled} planar graphs. 
Not even the growth constant $\gamma_u$ is known. 
It is however possible to show that  $\gamma_u > \gamma \approx 27.23$, where $\gamma$ is the constant for labelled planar graphs, and also the upper bound $\gamma_u < 30.06$ (see the discussion in \cite{GN09}).
Other natural open problems in the enumeration of planar graphs are: bipartite, triangle-free (more generally, $H$-free for fixed graph $H$), 4-connected and 5-regular.

A technical obstacle in our analysis is the size of the multivariate polynomials equations involved, in terms of the degree and the size of the coefficients, when one performs elimination.
To overcome this situation, we use the classical technique of evaluation and polynomial interpolation for elimination \cite{GG13}.
This results in quite large equations (two of them would need about 30 pages each to be printed), but remains within the capabilities of the \texttt{Maple} computer algebra system using a powerful computer.

We first find the equations for 3-connected 4-regular planar maps counted according to simple and double edges.
In order to guarantee correctness of our results, we need an upper bound on the degree of the resultant.
In our case the best upper bound we obtain is 160, hence we have to interpolate at 161 points to guarantee the correctness of the result.
In this way, we obtain the minimal polynomials satisfied by the generating functions $T_1(u,v)$ and $T_2(u,v)$ counting 3-connected 4-regular planar maps, as given in Appendix \ref{app:min-pol-T1T2}.

Once the equations for the $T_i(u,v)$ are obtained explicitly, further elimination yields the minimal polynomial $P(x,y)$ satisfied by the generating function $C^\bullet(x)=xC'(x)$ of vertex-rooted connected 4-regular planar graphs.
Since the polynomial $P$ is too large to be displayed in print, we provide a link to fully annotated \texttt{Maple} files\footnote{\url{https://requile.github.io/4-regular_planar_maple.zip}}, where all our computations can be reproduced.
In the sequel we refer to these files as ``the \texttt{Maple} sessions''.
From $P$ we compute the dominant singularity of $C(x)$, which turns out to be an algebraic number of degree~14, as shown in Theorem \ref{thm:main}.
Then we perform similar computations for 2-connected and 3-connected 4-regular planar maps, as well as for 4-regular simple maps (in this last case the minimal polynomials are small enough to be reproduced in Appendix \ref{app:min-pol-maps}).
These results together with the corresponding asymptotic estimates are the content of the following three theorems.

\newpage
Our main result is an estimate for the number of connected and arbitrary 4-regular planar graphs.
\begin{theorem}\label{thm:main}
(a) The number $c_n$ of connected 4-regular labelled planar graphs is asymptotically
\begin{equation*}
	c_n \sim c \cdot n^{-7/2}\cdot \gamma^{n}\cdot n!, \quad
 	\text{with } c \approx 0.0013911 \text{ and } \gamma=\rho^{-1} \approx 4.10228,
\end{equation*}
where $\rho \approx 0.24377$ is the smallest positive root of
\begin{equation}\label{poly:gamma}
	\begin{split}
	& 12397455648000x^{14} + 99179645184000x^{13} - 263210377713408x^{11} \\
	& + 4123379191922784x^{10} - 1230249287613888x^9 \\
	& - 18655766288483533x^8 + 51831438989552290x^7 \\
	& + 97598878903661028x^6 + 620596059256280x^5 \\
	& + 15894289357702528x^4 - 63729042783408384x^3 \\
	& - 66418928650596352x^2 + 64476004593270784x \\
	& - 109267739753840648 = 0.
	\end{split}
\end{equation}
(b) The number $g_n$ of 4-regular labelled planar graphs with $n$ vertices is asymptotically 	
\begin{equation*}
	g_n \sim g\cdot n^{-7/2}\cdot \gamma^{n}\cdot n!,
\end{equation*}
where $g > 0$ is a constant and where $\gamma$ is as in (a).
\end{theorem}

Although we cannot determine the constant $g$ in the previous statement analytically, its existence can be shown as follows. 
Since the class of 4-regular planar graphs is closed under disjoint unions, it follows from a standard argument that the number of connected components in  the class is asymptotically distributed like $1+X$, where $X$ is a Poisson random variable with parameter $\lambda = C(\rho)$. 
Hence the probability that a random graph is connected tends to $p =e^{-\lambda}$ as  $n\to\infty$.
We cannot compute $C(\rho)$ since we only have access to the derivative $C'(x)$, which is an algebraic function of degree 29 as explained in Section \ref{sec:asymptotic}.
If the algebraic curve defined from the minimal polynomial satisfied by $C'(x)$ were rational (i.e. of genus 0), one could find a rational parametrization and integrate it.
But it actually has genus 30, which prevents us from integrating it in that way. 

Instead, to estimate $p$ we use $c_n/g_n$ for $n$ fixed and as large as we can. 
For instance, we obtain $c_{500}/g_{500}\approx 0.999993$, and with this value we have $g=c/p\approx 0.00139$ (see the corresponding \texttt{Maple} session). 
We only know that the rate of convergence of $c_n/g_n$ is of order $O(1/n)$, because of general principles of analytic combinatorics, and thus we are unable to establish a confidence interval.
This value for $p$ close to $1$ makes sense since the smallest connected component is the graph of the octahedron, which is relatively large.
By contrast, the probability that a random planar graph is connected tends to $0.96325$ (see \cite{GN09}), while for a random cubic planar graph it is $0.999397$ (see \cite{NRR19+}).

Our next result are estimates for the number of 3- and 2-connected 4-regular planar graphs.
\begin{theorem}\label{thm:2-3-conn}
(a) The number $t_n$ of 3-connected 4-regular labelled planar graphs  with $n$ vertices  is asymptotically
\begin{equation*}
	t_n \sim t \cdot n^{-7/2}\cdot (\gamma_3)^{n}\cdot n!,
 	\text{with } t \approx 0.0012070 \text{ and } \gamma_3=\tau^{-1} \approx 4.08978,
\end{equation*}
where $\tau = \frac{88- 12\sqrt{21}}{135} \approx 0.24451$.

\medskip

(b) 	The number $b_n$ of 2-connected 4-regular labelled planar graphs  with $n$ vertices  is asymptotically
\begin{equation*}
	b_n \sim b \cdot n^{-7/2}\cdot (\gamma_2)^{n}\cdot n!, \text{with } b \approx 0.0000575832 \text{ and } \gamma_2=\beta^{-1} \approx 4.10175,
\end{equation*}
where $\beta \approx 0.2437981094$ is the smallest positive root of
\begin{equation}\label{poly:beta}
	\begin{split}
		& 12397455648000x^{11} + 24794911296000x^{10} - 148769467776000x^9 \\
		& + 1125304654862592x^8 - 451035134375328x^7 \\
		& - 7923244598779392x^6 + 38505114557935859x^5 \\
		& - 67113688868067728x^4 + 70322996382137760x^3 \\
		& - 43445179814077952x^2 + 12857755940483072x \\
		& - 1365846746923008 = 0.
	\end{split}
\end{equation}
\end{theorem}

The last result deals with the number of simple 4-regular maps. 
The enumeration of (non-necessarily simple) 4-regular maps is rather direct, since they are in bijection with arbitrary maps \cite[Chapter 5]{handbook15}. 
But forbidding loops and multiple edges makes the problem much more challenging.

\begin{theorem}\label{thm:maps}
(a) The number $u_n$ of 4-regular simple maps  with $n$ vertices  is asymptotically
\begin{equation*}
    u_n \sim s \cdot n^{-5/2}\cdot \sigma^{-n},
    \text{with } s \approx 0.016360 \text{ and } \sigma^{-1} \approx 4.13146,
\end{equation*}
where $\sigma \approx 0.24204$ is the smallest positive root of
\begin{equation}\label{poly:sigma}
	432x^8 + 448x^7 - 852x^6 + 588x^5 - 72x^4 - 504x^3 + 135x^2 + 108x - 27 = 0.
\end{equation}
(b) The number $h_n$ of 2-connected 4-regular simple maps is asymptotically
\begin{equation*}
	h_n \sim h \cdot n^{-5/2}\cdot \eta^{-n},  \text{with } h \approx 0.014477\text{ and } \eta^{-1} \approx 4.122915,
\end{equation*}
where $\eta \approx 0.24255$ is the smallest positive root of
\begin{equation}\label{poly:sigma}
    108x^6 + 4x^5 - 136x^4 + 344x^3 - 425x^2 + 196x - 27 = 0.
\end{equation}
\end{theorem}

\paragraph{Remark.}

The constants involved in the previous theorems are polynomial functions in the dominant singularity of their respective generating function with coefficients in $\mathbb{Q}[\pi]$ (this is due to the $\Gamma$ function, see Lemma \ref{lem:analytic} in Section \ref{sec:asymptotic}).
From those polynomials, and using the polynomial equations satisfied by the dominant singularities, one can also write the constants implicitely as solutions of minimal polynomials in $\mathbb{Z}[\pi]$.
At the exception of constant $g$, the minimal polynomial of each constant is given in the corresponding file in the \texttt{Maple} sessions, or in separate files for $b$ and $c$.
In addition, all the polynomials shown are irreducible and their integer coefficients have no common factor.

\medskip

The rest of the paper is organised as follows.
In Section \ref{sec:prelim} we recall the basic definitions on planar graphs and maps, then on algebraic generating functions.
In Sections \ref{sec:quadr} and \ref{sec:3-conn} we recall first the various combinatorial objects introduced in \cite{NRR19} and then the equations satisfied by the associated generated functions.
Then by elimination we find the minimal polynomials of quadrangulations and 3-connected 4-regular maps.
In Section \ref{sec:graphs} we use the results of the previous section to compute minimal polynomials for 4-regular planar graphs and maps.
Finally in Section \ref{sec:asymptotic}, after providing an analytic lemma, we obtain the asymptotic estimates for all the graphs and maps of interest.

%
%
\section{Preliminaries}\label{sec:prelim}

\subsection{Planar graphs and maps}

Throughout the paper, graphs are labelled and maps are unlabelled.
A graph is \emph{planar} if it admits an embedding on the plane without edge-crossings.
A planar map is an embedding of a planar multigraph up to orientation preserving homeomorphisms of the sphere.
It is simple if the underlying graph is simple.
A planar map $M$ is always considered {\em rooted}: an edge $ab$ of $M$ is distinguished and given a direction from $a$ to $b$.
The vertex $a$ is the {\em root vertex} and the face on the right of $ab$ as the {\em root face}.
Any other face is called an {\em inner face} of $M$.
Vertices incident with the root face are called \emph{external} vertices.

A map in which every vertex (resp. face) has degree four is said to be {\em 4-regular} (resp. a {\em quadrangulation}).
By duality, quadrangulations are in bijection with 4-regular maps.
Notice that quadrangulations can have ``degenerate'' faces consisting of a double edge with an isthmus inside.
A quadrangulation with at least eight vertices is {\em irreducible} if every 4-cycle forms the boundary of a face.
Irreducible quadrangulations are known to be in bijection with 3-connected maps (see \cite{BGW02}).

The following concepts are taken from \cite{NRR19} (see also \cite{MS68}).
A \emph{diagonal} in a quadrangulation is a path of length two whose endpoints are external and the central point is internal.
If $uv$ is the root edge, then there are two kinds of diagonals, those incident with $u$ and those incident with $v$.
By planarity both cannot be present at the same time.
A vertex of degree two in a quadrangulation is called \emph{isolated} if it is not adjacent to another vertex of degree two.
An isolated vertex of degree two will be called a \emph{2-vertex}.
By duality, a 2-vertex becomes (in the corresponding 4-regular map) a face of degree two not incident with another face of degree two.
We call it a \emph{2-face}.
Furthermore, we say that an edge is \emph{in} a 2-face if it is on its boundary, and \emph{ordinary} otherwise.
Note that since the number of edges of a 4-regular map is even, so is the number of ordinary edges.

\subsection{Algebraic generating functions}

A power series $f(x)$ is \textit{algebraic} if it satisfies a polynomial equation of the form
\begin{equation*}
	P(f(x), x) = p_k(x)f(x)^k + p_{k-1}(x)f(x)^{k-1} + \cdots + p_1(x)f(x) + p_0(x) = 0,
\end{equation*}
where the $p_i$'s are polynomials in $x$.
If the polynomial $P(y,x)$ is irreducible then it is unique and is called the  minimal polynomial of $f(x)$.
An algebraic  power series $f(x)$ with non-negative coefficients is represented  as a branch of its minimal polynomial $P(y,x)$ in the positive quadrant passing through the origin.
This last condition represents the fact that there is no graph with an empty vertex set.
We call this branch the \emph{combinatorial branch}.
It defines an analytic function in a disk centered at the origin with positive radius of convergence $\rho$.
Since the coefficients of $f(x)$ are non-negative, it holds by Pringsheim's theorem \cite[Theorem IV.6]{FS09} that $\rho$ is a singularity of $f(x)$, called the dominant singularity.
In this paper $\rho$ will always be a branch-point $(f(\rho),\rho)$ of $P(y,x)$, that is, one of the common roots of
\begin{equation*}	
	\frac{\partial P}{\partial y}(y,x)=0, \qquad  P(y,x)=0,
\end{equation*}	
It is a positive root of a factor of the discriminant of $P(y,x)$ with respect to $y$ (see \cite[Section VII.7]{FS09}).
All the algebraic power series  $f(x)$ in this paper  admit a {Puiseux expansion} as $x\to\rho^-$ (i.e. $|x|<\rho$ and $x\to\rho$) of the form
\begin{equation*}
	f(x) = f_0 - f_2\left( 1 - \frac{x}{\rho} \right) + f_3\left( 1- \frac{x}{\rho} \right)^{3/2} + O\left( \left( 1- \frac{x}{\rho} \right)^2 \right),
\end{equation*}
which is the local expansion of $f(x)$ near $\rho^-$ corresponding to the combinatorial branch of $P(y,x)$.
Note that $f_0$, $f_2$ and $f_3$ are algebraic functions in $\rho$ and are thus algebraic constants as $\rho$ also is.
Furthermore, we always have $f_0,f_2,f_3>0$.

Using the Newton polygon algorithm (see \cite[Section 6.3]{KP11}), one can compute any coefficent of this expansion exactly, i.e. in closed form or as root of a given polynomial, in time polynomial in the degree of $f(x)$.
This algorithm has been implemented as the function \texttt{puiseux} in the \texttt{Maple} package \texttt{algcurves} and this is what
we use in Section \ref{sec:asymptotic}.

%
%
\section{Equations for quadrangulations}\label{sec:quadr}

We follow the combinatorial scheme introduced in \cite[Section 2]{NRR19}, which we summarize here.
Following Mullin and Schellenberg in \cite{MS68}, we partition simple quadrangulations into three families:
\begin{itemize}
    \item[(S1)] The quadrangulation consisting of a single quadrangle.

    \item[(S2)] Quadrangulations containing a diagonal incident with the root vertex.
    By symmetry, they are in bijection with quadrangulations containing a diagonal not incident with the root vertex.
    Each of those two classes can be partitioned into three sub-classes $\cN_i$ for $i = 0,1,2$, according to the number $i$ of external 2-vertices.

    \item[(S3)] Quadrangulations obtained from an irreducible quadrangulation by possibly replacing each internal face with a simple quadrangulation.
    We denote this family by $\cR$.
\end{itemize}

We use simple quadrangulations to obtain generating functions for general quadrangulations encoding $2$-vertices \cite[Section 2.2]{NRR19}.
This is done in two steps.
First, we obtain equations for quadrangulations of the 2-cycle (see \cite[Lemma 2.3]{NRR19}).
We denote by $\cA = \cA_0 \cup \cA_1$ the quadrangulations of a 2-cycle, where $\cA_1$ are those whose root vertex is a 2-vertex (by symmetry, they are in bijection with those in which the other external vertex is a 2-vertex), and $\cA_0$ are those without external 2-vertices.

Finally we obtain equations for arbitrary quadrangulations $\cB$.
We decompose $\cB = \cB_0 \cup \cB^*_0 \cup \cB_1$, where $\cB_1 $ are those in which the root edge is incident with exactly one 2-vertex, and $\cB_0 \cup \cB^*_0$ are those in which the root edge is not incident with a 2-vertex.
Furthermore, $\cB_0^*$ are the quadrangulations obtained by replacing one of the two edges incident with the root edge in the single quadrangle with a quadrangulation of the 2-cycle.

\paragraph{Irreducible quadrangulations.}

We use \cite[Equation (9)]{BGW02}.
Let $s_n$ be the number of irreducible quadrangulations with $n$ inner faces.
Then the associated generating function $S(y) = \sum_{n\ge 0} s_n y^n$ satisfies the following implicit system of rational equations
\begin{align*}
    S(y) = \ds\frac{2y}{1 + y} - y - \frac{U(y)^2}{y(1 + 2U(y))^3}, \qquad
    U(y) = y(1 + U(y))^2.
\end{align*}

Eliminating $U(y)$ from the above system and factorising gives us a polynomial satisfied by $S(y)$.
Then by expanding the roots of each factor in series of $y$ near zero, one can check that the minimal polynomial of $R(y)$ is given by
\begin{equation}\label{poly:R}
    \begin{split}
        P_S(S(y),y) = \,\,
        & (y^5 + 8y^4 + 25y^3 + 38y^2 + 28y + 8)S(y)^2 \\
        & + \, (2y^6 + 12y^5 + 20y^4 + 10y^3 - 5y^2 - 4y + 1)S(y) \\
        & + \, (y^7 + 4y^6 - y^5) = 0.
    \end{split}
\end{equation}

\paragraph{From irreducible quadrangulations to simple quadrangulations.}

We use variables $s$ and $t$ to mark inner faces and $2$-vertices, respectively.
We write $N_i = N_i(s,t)$ ($i=0,1,2$) for the generating function of the subclass $\cN_i$ counting both the number of inner faces and 2-vertices.
We denote by $R = R(s,t)$ the generating function associated to~$\cR$.
By rewriting the system (1) in \cite{NRR19} using these variables we get:
\begin{equation}\label{sys:RN}
    \begin{array}{rll}
		y & = & s + 2\widetilde{N} + S(y), \\
		t^2\widetilde{N} & = & t^2N_0 + tN_1 + N_2, \\
		tN_0 & = &(\widetilde{N} + R) \left(t(\widetilde{N} + R + N_0) + \frac{N_1}{2}\right), \\
 		N_1 & = & 2st\left(\widetilde{N} + R + N_0 + \frac{N_1}{2}\right), \\
		N_2 & = & s^2t^3 + st\left(  \frac{N_1}{2}+N_2\right),
    \end{array}
\end{equation}
where $y$ is a function of $s$ and $t$.

\paragraph{From simple quadrangulations to general quadrangulations.}

We use $z$ and $w$ to mark inner faces and $2$-vertices, respectively.
In the following system of equations, variables $s$ and $t$ are considered as functions of $z$ and $w$.
As in \eqref{sys:RN}, we write $N_i = N_i(s,t)$ ($i=0,1,2$).
We denote by $A_j = A_j(s,t)$ the generating function of the family $\cA_j$ ($j=0,1$).

The equations in this context are as follows; see the details in \cite{NRR19}. In particular, the topmost equation in page 365, together with Equation (2) and the equations in Lemmas 2.3 and 2.4:
\begin{equation}\label{sys:A0}
    \begin{array}{rll}
    	s &=& z(1 + \widetilde{A})^2, \\
    	t(1 + \widetilde{A})^2 &=& w + 2\widetilde{A} + \widetilde{A}^2, \\
   		Q_0 &=& s(2N_0 + N_1+ S(y)) + (2\widetilde{A} + \widetilde{A}^2)Q_1, \\
    	tQ_1 &=& N_1 + 2N_2, \\
    	E &=& z(1 + \widetilde{A})^4 - 4z\widetilde{A}^2  + 4zw\widetilde{A}^2, \\
    	w\widetilde{A} &=& wA_0 + 2A_1, \\
    	w\widehat{A} &=& wA_0 +(1+w)A_1, \\
    	A_0 &=& 2z\widetilde{A}(1 + \widehat{A}) \\
    	&&+ z(Q_0 + Q_1 + E + 2z(w - 1)\widetilde{A} + 2z(1 - w)\widetilde{A}^2), \\
    	A_1 &=& zw(1 + \widehat{A}).
    \end{array}
\end{equation}
Then (see Lemma 2.5. in \cite{NRR19}) we have that
\begin{align}
	& B_0 = 2z(1 + \widehat{A})(1 + \widehat{A} - A_1) + z(Q_0 + E - 2zw\widetilde{A}^2  - 2z\widetilde{A}), \label{eq:B0} \\
	& B_1 = 2z(1 + \widehat{A})A_1 + zw(Q_1 + 2z\widetilde{A}^2), \label{eq:B1} \\
	& B_0^* = 2z^2\widetilde{A}. \label{eq:B0star}
\end{align}

We define next three systems of algebraic equations:
\begin{equation*}	
	{\mathbf S}_0 = \eqref{sys:RN} \cup \eqref{sys:A0} \cup \eqref{eq:B0},
	\qquad  {\mathbf S}_1 = \eqref{sys:RN} \cup \eqref{sys:A0} \cup \eqref{eq:B1}, \qquad
  	{\mathbf S}_0^* = \eqref{sys:RN} \cup \eqref{sys:A0} \cup \eqref{eq:B0star}.
\end{equation*}	
Each of these systems is composed of sixteen equations, eighteen variables, and is strongly connected.
By algebraic elimination (using the \texttt{Maple}  function \texttt{Groebner}), one can  obtain a unique polynomial equation in any three chosen variables.
We  obtain the following three polynomials which are respectively of degree 2 in $B_0$, $B_1$ and $B_0^\ast$
\begin{align}
	& P_{B_0}(B_0,z,w) =p_{0,0}(z,w)+p_{1,0}(z,w)B_0+p_{2,0}(z,w)B_0^2,\label{poly:B0} \\
	& P_{B_1}(B_1,z,w) =p_{0,1}(z,w)+p_{1,1}(z,w)B_1+p_{2,1}(z,w)B_1^2,\label{poly:B1} \\
	& P_{B_0^*}(B_0^*,z,w) =p_{0,2}(z,w)+p_{1,2}(z,w)B_0^\ast+p_{2,2}(z,w)(B_0^\ast)^2,\label{poly:B0star}
\end{align}
where each coefficient $p_{i,j}(z,w),\,p_{i,j}^\ast(z,w)$ is a bivariate polynomial given in Appendix \ref{app:min-pol-B}.
These polynomials are obtained after the elimination process by factoring the resulting polynomials and choosing the right factor in each case (in all cases we get only one candidate with non-negative integer coefficients).

%
%
\section{Equations for 3-connected 4-regular maps}\label{sec:3-conn}

We shortly describe how to obtain a combinatorial decomposition scheme in order to deduce equations for 3-connected 4-regular maps.
More details are given in \cite[Section 3]{NRR19}.

The class $\cM$ of $4$-regular maps can be decomposed into $\cM_0 \cup \cM^*_0\cup \cM_1 $, where $\cM_0 \cup \cM_0^*$ are 4-regular maps in which the root edge is not incident with a 2-face, $\cM_1$ are those for which the root edge is incident with exactly one 2-face, and $\cM_0^*$ are maps in which the root is one of the outer edges of a triple edge.
These classes are in bijection with the classes $\cB_0$, $\cB^*_0$ and $\cB_1$ from the previous section, as the dual of a quadrangulation with $\ell$ 2-vertices is a 4-regular map with $\ell$ 2-faces.

We denote by $M_0(q,w)=M_0$, $M_1(q,w)=M_1$ and $M_1^{\ast}(q,w)=M_0^{\ast}$ the associated generating functions, where variables $w$ and $q$ mark 2-faces and ordinary edges, respectively.
Observe that when setting $w = q$, one recovers the enumeration of 4-regular maps according to half the number of edges.
Due to the bijection with quadrangulations it follows that
\begin{align}\label{eq: M's}
    M_0(q,w) = B_0(q,w/q), &
    & M_1(q,w) = B_1(q,w/q), &
    & M_0^*(q,w) = B_0^*(q,w/q).
\end{align}

The next step is to decompose the previous classes.
Given a map $M$, let $M^-$ be the map obtained by removing the root edge $st$, whose endpoints $s,t$ are called the \emph{poles} of $M$.
As shown in \cite[Lemma 3.1]{NRR19} we have
\begin{equation*}
	\cM_0 = \cL \cup \cS_0\cup \cP_0\cup \cH, \qquad  \cM_1 = \cS_1 \cup \cP_1\cup \cF\cup \overline{\cF},
\end{equation*}
where
\begin{itemize}
	\item $\cL$ are maps in which the root-edge is a loop.

	\item $\cS = \cS_0 \cup \cS_1$ are \emph{series} maps: $M^-$ is connected and there is an edge in $M^-$ that separates the poles.
	The index $i=0,1$ refers to the number of 2-faces incident with the root edge.
	
	\item $\cP = \cP_0 \cup \cP_1$ are \emph{parallel} maps: $M^-$ is connected, there is no edge in $M^-$ separating the poles, and either $st$ is an edge of $M^-$ or $M - \{s,t\}$ is disconnected.
	The index $i=0,1$ has the same meaning as in the previous class.

	\item $\cH$ are \emph{polyhedral} maps: they are obtained by considering a 3-connected 4-regular map $C$ (called the {\em core}) rooted at a simple edge and possibly replacing every non-root edge of $C$ with a map in $\cM$.
	
	\item $\cF$ (resp. $\overline{\cF}$)  are maps $M$ such that the face to the right (resp. to the left) of the root-edge is a 2-face, and such that $M - \{s,t\}$ is connected.
\end{itemize}

Let $L$, $S_0$, $S_1$, $P_0$, $P_1$, $F$ and $\overline{F}$ be the generating functions of the corresponding families each in terms of the variables $q$ and $w$.
In the case of  $S_1$ and $P_1$, we count the number of 2-faces minus one (instead of the total number of 2-faces) and the number of ordinary edges plus two.
In both $F$ and $\overline{F}$ we count the number of 2-faces minus one and the number of ordinary edges (and in particular $F=\overline{F}$ by symmetry).

Finally, $\cT_1$ and $\cT_2$ are the classes of 3-connected 4-regular maps rooted at a simple and at a double edge, respectively.
We denote by $T_1(u,v)$ and $T_2(u,v)$ the corresponding generating functions, where $u$ and $v$ respectively mark ordinary edges and 2-faces.
The main purpose of this section is to obtain the minimal polynomials for both $T_1$ and $T_2$.

\paragraph{The system for $T_1(u,v)$.}

The following is the system $(5)$ from \cite[Lemma 3.2]{NRR19}.
We include Equations \eqref{poly:B0}-\eqref{poly:B0star} which define implicitly $M_0$, $M_1$ and $M_0^\ast$ in terms of $q$ and $w$.

\begin{equation}\label{sys:T1}
    \begin{array}{rll}
    	(1 + D)H &=& T_1(u,v), \\
    	u &=& q(1 + D)^2, \\
    	v &=& w + q(2D + D^2 + F), \\
   		M_0 &=& S_0 + P_0 + L + H , \\
    	qM_1 &=& w(S_1 + P_1 + 2qF), \\
    	M_0^* &=& 2q^2D, \\
    	L &=& 2q(1 + D - L) + L(w+q)L, \\
    	S_0 &=& D(D - S_0 - S_1) - L^2/2, \\
    	S_1 &=& L^2/2, \\
    	P_0 &=& q^2(1 + D + D^2 + D^3) + 2qDF, \\
    	P_1 &=& 2q^2D^2, \\
    	0 &=& P_{B_0}(M_0,q,w/q), \\
    	0 &=& P_{B_1}(M_1,q,w/q), \\
    	0 &=& P_{B_0^\ast}(M_0^*,q,w/q).
    \end{array}
\end{equation}
Observe that all the generating functions (including $u$ and $v$) are functions of $q$ and $w$.

\paragraph{The system for $T_2(u,v)$.}

For maps in $\cH$ the root of the core  is a simple edge,  hence we need to modify slightly the system of equations \eqref{sys:T1} in order to get an equation for $T_2(u,v)$.
We adapt \cite[Lemma 4.1]{NRR19} to the map setting and obtain
\begin{equation*}
	\cF = \cS_2 \cup \cH_2,
\end{equation*}
where $\cS_2$ are networks in $\cF$ such that after removing the two poles there is a cut vertex, while $\cH_2$ are networks in $\cF$ whose root edge is incident to a 2-face.
We denote by $S_2$ and $H_2$ the corresponding generating functions.

Equations for $F$, $S_2$ and $H_2$ are deduced in \cite[Equation $(7)$]{NRR19}.
Combining them with the decomposition explained for $T_1$ we get the following system of equations.
We observe that here the minimal polynomial of $B_0$ it is not needed.
Similarly to \eqref{sys:T1}, all the  generating functions involved in \eqref{sys:T2} depend on $q$ and $w$ but we omit to write the arguments.
\begin{equation}\label{sys:T2}
    \begin{array}{rll}
    	vH_2 &=& T_2(u,v), \\
        u &=& q(1 + D)^2, \\
        v &=& w + q(2D + D^2 )+ F, \\    	
        qM_1 &=& w(S_1 + P_1 + 2qF), \\
        M_0^\ast &=& 2q^2D, \\
        L &=& 2q(1 + D - L) + (w+q)L,\\
        S_0 &=& D(D - S_0 - S_1) - L^2/2, \\
        S_1 &=& L^2/2, \\
        P_0 &=&  q^2(1 + D + D^2 + D^3) + 2qDF, \\
        P_1 &=& 2q^2D^2, \\
        F &=& S_2 + H_2, \\
        S_2 &=& (w + q(2D + D^2) + F)(w + q(2D + D^2) + F - S_2), \\
        0 &=& P_{B_1}(M_1,q,w/q), \\
        0 &=& P_{B_0^\ast}(M_0^*,q,w/q).
    \end{array}
\end{equation}

\subsection{The minimal polynomials of $T_1$ and $T_2$}

The next step is to compute the minimal polynomials $P_{T_1}(T_1,u,v)$ and $P_{T_2}(T_2,u,v)$ defining implicitly $T_1$ and $T_2$ as functions of $u$ and $v$.
In what follows, we present the method used to obtain $P_{T_1}$, which is based on evaluation and interpolation.
The same method is then used to compute $P_{T_2}$.

\paragraph{Evaluation and interpolation for $P_{T_1}$.}

First, from  \eqref{sys:T1} one can eliminate variables $M_0$, $M_1$, $M_0^*$, $S_0$, $S_1$, $P_0$, $P_1$, $D$, $H$ and $F$ and obtain  irreducible polynomial equations in the corresponding variables:
\begin{align*}
    & Q_{T_1}(T_1,q,w) = 0, &
    & Q_u(u,q,w) = 0, &
    & Q_v(v,q,w) = 0.
\end{align*}
Notice that these equations define implicitly $T_1$, $u$ and $v$ as functions of $q$ and $w$.

From there we compute the resultant of $Q_u$ and $Q_v$ with respect to $w$ and find its unique \textit{combinatorial factor} $Q_1(u,v,q)$, that is, the one whose Taylor expansion at $q=0$ has non-negative integer coefficients.
The polynomial $Q_1$ has degree 10 in both $u$ and $v$, and 16 in $q$.
We  compute similarly $Q_2(u,T_1,q)$, the unique combinatorial factor of the resultant of $Q_{T_1}$ and $Q_v$ with respect to $w$.
It has degree 10 in $u$, 20 in $T_1$, and 16 in $q$.
This gives the system:
\begin{equation}\label{sys:eval_interp}
	Q_1(u,v,q) = 0, \qquad
    Q_2(T_1,u,q) = 0.
\end{equation}
If we could now compute directly the resultant of $Q_1$ and $Q_2$ with respect to $q$, this would lead to a polynomial equation $R(T_1,u,v) = 0$ having $P_{T_1}$ as one of its factors, and we would be done.
However, this seems to require too much computing time, even for a relatively powerful computer\footnote{The \texttt{Maple} computation ran for about a week on a machine Intel(R) Xeon(R) CPU E5-2687W v4 @ 3.30GHz$\times$48.}.
Both $Q_1$ and $Q_2$ are dense and their coefficients in $q$ are bivariate polynomials in $\mathbb{Z}[u,v]$ and $\mathbb{Z}[u,T_1]$, respectively, of high total degree.
Instead we proceed indirectly using evaluation and polynomial interpolation.

Let $d$ be the degree of $v$ in $P_{T_1}$.
We will evaluate the system \eqref{sys:eval_interp} at $d+1$ different values $v=k_0,k_1,\ldots,k_d$ and compute the resultant with respect to $q$ of each evaluation.
For any fixed and sufficiently small integer $k>0$, we can compute effectively the resultant of $Q_1(u,k,q)$ and $Q_2(T_1,u,q)$ with respect to $q$.
Notice that if the leading coefficient of $Q_1(u,v,q)$ is not divisible by $(v-k)$, then this resultant is precisely $R(T_1,u,k)$.
Expanding each factor of $R(T_1,u,k)$ at $u=0$ allows us to find the combinatorial one\footnote{
    In this case, the coefficients of $T_1(u)$ count 3-connected 4-regular planar graphs rooted at a simple edge and whose double edges are weighted by $k$.
    The reason why we chose to evaluate at $v$ instead of $u$, or why we do not proceed by nested evaluation and interpolation, e.g. at $v$ then at $u$, is to guarantee the existence of such a combinatorial factor at each evaluation, which has in practice much smaller degrees than the resultant.
}, denoted by $F_k(T_1,u)$.
Following this process we obtain $d+1$ polynomials $F_{k_0},\ldots, F_{k_d}$ in $\mathbb{Z}[T_1,u]$.
Finally interpolating those $d+1$ points yields a unique polynomial in $\mathbb{Z}[T_1,u][v]\simeq \mathbb{Z}[T_1,u,v]$ of degree $d$ in $v$.
This polynomial is exactly $P_{T_1}(T_1,u,v)$, because for each $k = k_0, \ldots, k_d$ the coefficients of $F_k(T_1,u)$ are integers with no common divisor.

We now derive an effective upper bound for $d$.
The resultant $R(T_1,u,v)$ can be seen as the determinant of the {Sylvester matrix} associated with the system \eqref{sys:eval_interp}.
This matrix has size 32, the sum of the degrees of $q$ in $Q_1$ and $Q_2$.
By construction, the coefficients in the first 16 columns are polynomials in $\mathbb{Z}[u,v]$, each of degree at most 10 in $v$, while in the rest of the columns they are polynomials in $\mathbb{Z}[T_1,u]$.
Each monomial of $R$ is hence a product of 32 bivariate polynomials, exactly 16 of which contain a term in $v$ and of degree at most 10.
Thus as a factor of $R$, $P_{T_1}$ has degree at most $160$ in $v$.

After proceeding by evaluation and interpolation for $v = 1, 2,\ldots, 161$, checking at every step that the leading coefficient in $q$ of $Q_1(u,v,q)$ is not divisible by $v-k$, we obtain $P_{T_1}$.
It is monic and has degree 8 in $T_1$, 16 in $u$ and 8 in $v$, as follows:
\begin{equation}\label{poly:T1}
    P_{T_1}(T_1(u,v),u,v) = \sum_{i=0}^8 t_{i,1}(u,v)\cdot T_1(u,v)^i,
\end{equation}
where the $t_{i,1}(u,v)$ are polynomials in $u$ and $v$ given in Appendix \ref{app:min-pol-T1T2}.

\paragraph{Evaluation and interpolation for $P_{T_2}$.}

As argued in \cite[Section 4]{NRR19}, the generating function $T_2(u,v)$ is algebraic and satisfies a minimal polynomial equation denoted by $P_{T_2}(T_2,u,v) = 0$.
We will now show that $P_{T_2}$ has in fact degree at most 8 in $T_2$.
Recall first that $(\partial/\partial u) T_2(u,v) = (\partial/\partial v) T_1(u,v)$.
Second, because $P_{T_1}$ is monic and of degree 8 in $T_1$ then $P_{T_1,v}$ the derivative of $P_{T_1}$ with respect to $v$ has degree 7 in $T_1$ and is linear in $(\partial/\partial v) T_1$.
The minimal polynomial of $(\partial/\partial v) T_1$ is then a factor of the resultant of $P_{T_1}$ and $P_{T_1,v}$ with respect to $T_1$, whose degree in $(\partial/\partial v) T_1$ is at most seven\footnote{This can be seen by considering the resultant as the determinant of the Sylvester matrix associated to the system.}.
This means that $(\partial/\partial u) T_2$ is algebraic of degree at most 7, and hence $P_{T_2}(T_2,u,v)$ has degree most 8 in $T_2$.

We now proceed similarly to $T_1$ in order to compute $P_{T_2}$.
First, we eliminate from \eqref{sys:T2} to obtain an equation $Q_{T_2}(T_2,q,w) = 0$ that defines $T_2$ as a function of $q$ and $w$.
Notice that the equations for $u$ and for $v$ are exactly the same in both systems \eqref{sys:T1} and \eqref{sys:T2}, so that the two equations $Q_u = 0$ and $Q_v = 0$ computed for $T_1$ are the same, as well as the combinatorial factor $Q_1$ of their resultant with respect to $w$.
Then we compute $Q_3$, the combinatorial factor of the resultant of $Q_{T_2}$ and $Q_v$ with respect to $w$.
It has degree 10 in $T_2$, 27 in $v$ and 16 in $q$.

Notice that the required number of evaluations $v=1,2,\ldots$ of the system $\{Q_1,Q_3\}$ is at most $433$, while for the evaluations $u=1,2,\ldots$ we would need at most $161$.
We thus opt for the evaluations $u=k$, with $k=1,2,\ldots,161$, as follows.
We first verify that the leading coefficient of $q$ in $Q_1(u,v,q)$ is not divisible by $u-k$, and then compute $R_k(T_2,v)$ the resultant of the evaluation of $\{Q_1,Q_3\}$ at $u=k$.
It has five different factors.
Three of which cannot equal to zero.
And among the remaining two, one has degree 144 in $T_2$, and one has degree 8.
But because $P_{T_2}(T_2,u,v)$ has degree at most 8 in $T_2$, so does any of its evaluation at $u=1,\ldots,161$.
Thus $F_k(T_2,v)$, the combinatorial factor of $R_k(T_2,v)$ is the one of degree 8 in $T_2$.
$P_{T_2}$ is finally given as the interpolation of the $F_k(T_2,v)$'s, for $k=1,\ldots,161$.
It has degree 8 in $T_2$, 8 in $u$ and 13 in $v$, as follows:
\begin{equation}\label{poly:T2}
	P_{T_2}(T_2(u,v),u,v) = \sum_{i=0}^8 t_{i,2}(u,v)\cdot T_2(u,v)^i,
\end{equation}
where the $t_{i,2}(u,v)$'s are given in Appendix \ref{app:min-pol-T1T2}.\footnote{
    The two evaluation-interpolations for $T_1$ and $T_2$ took in total around 22 hours and 50 minutes to run on \texttt{Maple 2021} with a personal computer (8Go DDR3 RAM, Intel(R) Core(TM) i5-3550 CPU @ 3.30GHz$\times$4), by using the libary \texttt{CurveFitting} and the function \texttt{PolynomialInterpolation}.
	Both are included in the accompanying \texttt{Maple} sessions.
	In order to avoid any uncertainty due to the use randomness in \texttt{Maple}, all computations were done setting the environmental variable \texttt{\_EnvProbabilistic} to zero at the begining of each of the \texttt{Maple} sessions (this does not, it seems, have a significant impact on the execution time).
}

%
%
\section{Counting 4-regular planar graphs and simple maps}\label{sec:graphs}

The final step is to adapt the equations introduced in the previous section for graphs instead of maps. We follow the definitions and notation of \cite[Section 4]{NRR19}.

A \emph{network} is a connected 4-regular multigraph $G$ with an ordered pair of adjacent vertices $(s, t)$ such that the graph obtained by removing the edge $st$ is simple.
Vertices $s$ and $t$ are called the \emph{poles} of the network.
We  define several classes of networks, similar to the classes of maps introduced in the previous section.
We use the same letters, but they now represent classes of labelled \emph{graphs} instead  of maps.
No confusion should arise since in this section we deal only with graphs.
\begin{itemize}
    \item $\cD$ is the class of all networks.

    \item $\cL, \cS, \cP$ correspond as before to loop, series and parallel networks. We do not need to distinguish between $\cS_0$ and $\cS_1$ and between $\cP_0$ and $\cP_1$.

    \item $\cF$ is the class of networks in which the root edge has multiplicity exactly two and removing the poles does not disconnect the graph.

    \item $\cS_2$ are networks in $\cF$ such that after removing the two poles there is a cut vertex.

    \item $\cH = \cH_1 \cup \cH_2$ are $h$-networks: in $\cH_1$ the root edge is simple and in $\cH_2$ it is double.
\end{itemize}
The generating functions of networks are of the exponential type in the variable $x$ marking vertices.
We use letters $D$, $L$, $S$, $P$, $F$, $H_1$ and $H_2$ to denote the EGFs associated to the corresponding network class.

We next define the  generating functions $T^{(i)}(x,u,v)$ of 3-connected 4-regular planar multigraphs rooted at a directed edge, where $i=1,2$ indicates the multiplicity of the root, $x$ marks vertices and $u,v$ mark, respectively, half the number of simple edges and  double edges.
They are easily obtained from the generating functions of 3-connected 4-regular maps computed in the previous section, as follows:
\begin{align}\label{eq:TT}
    & T^{(i)}(x,u,v) = \frac{1}{2}T_i(u^2x,vx),
    & i\in \{1,2\}.
\end{align}
where the division by two encodes the choice of the root face.

\subsection{Connected 4-regular planar graphs}

The following equations are from \cite[Lemma 4.2]{NRR19}.
We denote by $C^{\bullet}(x) = xC'(x)$ the exponential generating function of connected 4-regular planar graphs rooted at a vertex.
\begin{equation}\label{sys:Cdot}
    \begin{array}{rll}
        4C^{\bullet} &=& D - L - L^2 - F  - x^2D^2/2, \\
        D &=& L + S + P + H_1 + F, \\
        L &=& \frac{x}{2}(D - L), \\
        S &=& (D - S)D, \\
        P &=& x^2\left(D^2/2 + D^3/6\right) + FD, \\
        F &=& S_2 + H_2, \\
        xS_2 &=& x^2v(x^2v - S_2), \\
        uH_1 &=& T^{(1)}, \\
        vH_2 &=& T^{(2)}, \\
        u &=& 1 + D, \\
        2x^2v &=& x^2(2D+ D^2) + 2F, \\
        0 &=& P_{T_1}\left(\frac{1}{2}T^{(1)},u^2x,vx\right), \\
        0 &=& P_{T_2}\left(\frac{1}{2}T^{(2)},u^2x,vx\right).
    \end{array}
\end{equation}

\subsection{2-connected 4-regular planar graphs}

Equations for 2-connected 4-regular planar graphs are very similar, they only differ in the fact that networks rooted at a loop will not appear in the recursive decomposition of a graph.
Hence, we just need to remove networks in $\cL$ from the equations.
Let $B^{\bullet}(x) = xB'(x)$ be the EGF of 2-connected 4-regular planar graphs rooted at a vertex, where again variable $x$ marks vertices.
The system of equations defining $B^{\bullet}(x)$ is given by
\begin{equation}\label{sys:Bdot}
    \begin{array}{rll}
        4B^{\bullet} &=& (D - F)  - x^2D(x)^2/2, \\
        D &=& S + P + H_1 + F, \\
        S &=& (D - S)D, \\
        P &=& x^2\left(D^2/2 + D^3/6\right) + FD, \\
        F &=& S_2 + H_2, \\
        xS_2 &=& x^2v(x^2v - S_2), \\
        uH_1 &=& T^{(1)}, \\
        vH_2 &=& T^{(2)}, \\
        u &=& 1 + D, \\
        2x^2v &=& x^2(2D+ D^2) + 2F, \\
        0 &=& P_{T_1}\left(\frac{1}{2}T^{(1)},u^2x,vx\right), \\
        0 &=& P_{T_2}\left(\frac{1}{2}T^{(2)},u^2x,vx\right).
    \end{array}
\end{equation}

\subsection{Simple 4-regular planar maps}\label{sec:simple-4regMap}

Let $M(x)$ be the (ordinary) generating function of 4-regular simple maps, where the variable $x$ marks vertices (vertices in maps are unlabelled).
As shown in \cite[Lemma 5.1]{NRR19}, $M(x)$ satisfies the following system of equations:
\begin{equation}\label{sys:M}
    \begin{array}{rll}
        M &=& D - L - L^2 - 3x^2D^2 - 2F, \\
        D &=& L + S + P + H_1 + 2F, \\
        L &=& 2x(D - L), \\
        S &=& D(D - S), \\
        P &=& x^2(3D^2 + D^3) + 2FD, \\
        F &=& S_2 + H_2/2, \\
        xS_2 &=& \left(x^2(2D + D^2) + F\right)(x^2(2D + D^2) + F - S_2), \\
        uH_1 &=& T_{1}, \\
        vH_2 &=& T_{2}, \\
        u &=& 1 + D, \\
        x^2v &=& x^2(2D + D^2) + F, \\
    	0 & = & P_{T_1}\left(T^{(1)},u^2x,vx\right), \\
        0 & = & P_{T_2}\left(T^{(2)},u^2x,vx\right).
    \end{array}
\end{equation}

\subsection{2-connected simple 4-regular planar maps}

Let $N(x)$ be the generating function counting 2-connected 4-regular simple maps.
This case was not considered in \cite{NRR19}, however it follows easily by removing loop networks in the previous system and we obtain
\begin{equation}\label{sys:N}
    \begin{array}{rll}
        N &=& D -3x^2D^2 - 2F, \\
        D &=& S + P + H_1 + 2F, \\
        S &=& D(D - S), \\
        P &=& x^2(3D^2 + D^3) + 2FD, \\
        F &=& S_2 + H_2/2, \\
        xS_2 &=& \left(x^2(2D + D^2) + F\right)(x^2(2D + D^2) + F - S_2), \\
        uH_1 &=& T_{1}, \\
        vH_2 &=& T_{2}, \\
        u &=& 1 + D, \\
        x^2v &=& x^2(2D + D^2) + F, \\
    	0 & = & P_{T_1}\left(T^{(1)},u^2x,vx\right), \\
        0 & = & P_{T_2}\left(T^{(2)},u^2x,vx\right).
    \end{array}
\end{equation}

%
%
\section{Asymptotic enumeration} \label{sec:asymptotic}

In this last section, we prove Theorems \ref{thm:main}, \ref{thm:2-3-conn} and \ref{thm:maps}.
For the sake of clarity, we omit certain computational details in the proof of Theorem \ref{thm:maps}, which can be found in the accompanying \texttt{Maple} sessions.
We first need the following analytic lemma.

\begin{lemma}\label{lem:analytic}
	Let $f(x)$ be an algebraic generating function with non-negative coefficients such that $f(0)=0$.
	Further assume that $f(x)$ admits a unique dominant singularity $\rho$ in the circle boarding the disk of convergence, and a Puiseux expansion as $x\to\rho^-$ of the form:
	\begin{equation}\label{puis:f}
	   f(x) = f_0 - f_2\left( 1- \frac{x}{\rho} \right) + f_3\left( 1- \frac{x}{\rho} \right)^{3/2} + O\left( \left( 1- \frac{x}{\rho} \right)^2 \right),
	\end{equation}
	with $f_0,f_2,f_3>0$.
	Then the coefficients of $f(x)$ verify the  asymptotic estimate
	\begin{equation*}
		[x^n]f(x)\sim \frac{3f_3}{4\sqrt{\pi}}\cdot n^{-5/2}\cdot \rho^{-n}, \qquad  \hbox{ as } {n\to\infty}.
	\end{equation*}	
\end{lemma}

\begin{proof}
	Since $f(x)$ is algebraic and $\rho$ is the unique singularity in the circle boarding the disk of convergence, one can  show by a classical compactness argument (see for instance the proof of \cite[Theorem 2.19]{D09}) that $f(x)$ is analytic in a $\Delta$-domain at $\rho$.
	We can then apply the \textit{transfer theorem} \cite[Corollary VI.1]{FS09} to the local representation \eqref{puis:f}, and deduce the  estimate as claimed, using  the relation $\Gamma(-3/2) = 4\sqrt{\pi}/3$.
\end{proof}

\paragraph{Proof of Theorem \ref{thm:main}.}

The system of equations \eqref{sys:Cdot} shows that $C^{\bullet} = C^{\bullet}(x)$ is an analytic function of $D = D(x)$.
This implies in particular that they both have the same singular behaviour. We first compute the equation satisfied  by $D$ from \eqref{sys:Cdot} minus the first equation.
After eliminating all the other variables, we obtain a polynomial in $D$ and $x$ with six factors.
The following three factors: $281474976710656$, $(1+D)^{48}$ and $(x+2)^{18}$ cannot be equal to zero.
We can also discard two other factors, one since its expansion at $x=0$ has constant term $-1/2$, different from zero, while the other admits an expansion of the form $\frac{1}{2}x^6 + \frac{1}{4}x^7 + O(x^8)$, which does not agree with the actual exponential generating function of networks starting with $ \frac{1}{2}x^6 + \frac{3}{4}x^7$.
Hence the minimal polynomial of $D$ must be the remaining factor of degree 29
\begin{equation}\label{poly:D}
    P_{D}(D(x),x) = \sum_{i=0}^{29} d_i(x)\cdot D(x)^i.
\end{equation}
The discriminant $p_D(x)$ of $P_{D}(D,x)$ with respect to $D$ has several irreducible factors and we have to locate the one having the dominant singularity as its root, which must be positive and less than 1.
Once we discard factors that do not have positive real roots less than 1, we have
\begin{equation*}	
	p_D(x) = f(x) g(x)^2 h(x)^3,
\end{equation*}	
where $f,g$ have respective degrees $155$ and $78$, and $h$ is the polynomial of degree 14 in the statement of Theorem \ref{thm:main}.
In order to rule out $g$ and $f$, let us first recall that the dominant singularity of all labelled planar graphs is  $\rho_1 \approx 0.0367$ \cite{GN09}.
The only candidate root for $g$ is $0.00021$, which can then be discarded because it is less than $\rho_1$.
The polynomial $f$ has two candidate solutions: one is $0.026$ and it is discarded for the same reason as before; the other one is $0.86$ and is discarded because it is larger than the singularity $\tau \approx 0.24451$ of 3-connected 4-regular graphs.
Hence the dominant singularity $\rho$ of $D(x)$ (and of $C(x)$ as argued above) is $\rho\approx 0.24377$ the smallest positive root of $h$.

To compute the minimal polynomial of $C^{\bullet}$, we eliminate all the other variables from \eqref{sys:Cdot} and obtain a polynomial equation in $D$, $C^{\bullet}$ and $x$.
We compute its resultant with $P_D$ with respect to $D$ to obtain a polynomial in $C^{\bullet}$ and $x$ only.
It has also six factors.
We can discard four of them as they trivially cannot be equal to zero, as before.
Another factor admits an expansion with constant term $7/8$, different from zero, and can be discarded too.
Finally the minimal polynomial of $C^{\bullet}$ is given by the remaining factor of degree 29
\begin{equation}\label{poly:Cdot}
    P_{C^{\bullet}}(C^{\bullet}(x),x) = \sum_{i=0}^{29} c_i(x)\cdot C^{\bullet}(x)^i.
\end{equation}
As a sanity check, one can expand the first terms and verify that they indeed count connected 4-regular planar graphs rooted at a vertex (see the corresponding \texttt{Maple} session and compare with Table 1 at the end of \cite{NRR19}, remembering that there are $n$ ways to root a graph with $n$ vertices).

From $P_{C^{\bullet}}(C^{\bullet},x)$ we compute the Puiseux expansion of $C^{\bullet}(x)$ at $x=\rho$ associated to the combinatorial branch, which turns out to be
\footnote{
	Note that this computation is done without using any numerical approximation (in particular that of $\rho$) and is thus exact (albeit hard to check by hand, i.e. without a computer algebra system, due to the size of the coefficients).
	In particular, the fact that $C^{\bullet}_1=0$ holds algebraically.
	On the other hand, the approximations can be computed using $\rho\approx 0.24377$, as the coefficients are algebraic functions of $\rho$, or from the minimal polynomials satisfied by the coefficients.
} 
\begin{equation}\label{eq:puis_C}
	C^{\bullet}(x) = C^{\bullet}_0 + C^{\bullet}_2X^2 + C^{\bullet}_3X^3 + O(X^4),
\end{equation}
where $X = \sqrt{1 - {x}/{\rho}}$, $C_0\approx 0.000057592$, $C_2\approx -0.00098931$ and $C_3\approx 0.0032877$.

To obtain the estimate for the coefficients of $C(x)$, we apply Lemma \ref{lem:analytic} to \eqref{eq:puis_C} and divide the resulting estimate of the coefficients of $C^{\bullet}(x)$ by $n$ since there are $n$ different ways to root a graph of size $n$ at a vertex.
By integrating \eqref{eq:puis_C} we obtain the Puiseux expansion $C(x) = C_0 + C_2X^2 + C_3X^3 + O(X^4)$.
However, the constant $C_0$ is undetermined after integration.
The estimate for the coefficients of $G(x)$ follows from $G(x) = \exp(C(x))$ and the corresponding Puiseux expansion
\begin{equation*}
	G(x) = G_0 + G_2X^2 + G_3X^3 + O(X^4).
\end{equation*}
Since $G_3 = e^{C_0}C_3$, this coefficient cannot be determined either.
As mentioned after the statement of Theorem \ref{thm:main}, we have estimated the constant $g=G_3/\Gamma(-3/2)$ from the first values of the coefficients $g_n$.
A similar situation occurs in \cite{NRR19+}.
There, we circumvented it by using the so-called ``dissymmetry theorem''.
For this, one needs in particular the generating function of \emph{unrooted} 3-connected 4-regular planar graphs, which means integrating $T_1(u,v)$ with respect to $u$.
We are not able to compute this integral; besides the fact that the size of the  equation defining $T_1(u,v)$ is rather large, it defines a curves of genus 1, hence it does not admit a rational parametrization.

\paragraph{Remark.}
Given the minimal polynomial of $xC'(x)$, it is possible in principle  to find linear
differential operators for $C(x)$ and for $G(x)$ and consequently linear
recursions with polynomial coefficients for the sequences of interest providing an effective formulation of Corollary 1.2 in \cite{NRR19}
(we are indebted to an anonymous referee for this remark). However the polynomial in Equation \eqref{eq:puis_C} is of degree 29 in $C^\bullet$ and 135 in $x$ with several integer coefficients having more than 40 digits, thus we have not tried to obtain explicit recurrences.

\paragraph{Proof of Theorem \ref{thm:2-3-conn}.}

We will now compute an estimate for the number of 3-connected 4-regular planar graphs.
By first plugging Equation \eqref{eq:TT} into the minimal polynomial of $T_1$, we obtain the minimal polynomial of $T^{(1)}(x,u,v)$.
Setting then $v=0$ and $u=1$, and taking the root edge into account, it is a simple matter to check that we obtain a polynomial satisfied by the generating function $T^{\bullet}(x)$ of 3-connected 4-regular planar graphs rooted at a vertex, namely
\begin{equation*}
	4T^{\bullet}(x) = T^{(1)}(x,1,0).
\end{equation*}
This polynomial is of the form
\begin{equation*}\label{poly:Tdot}
    P_{T^{\bullet}}(T^{\bullet}(x),x) = \sum_{i=0}^{8} t_i(x)\cdot T^{\bullet}(x)^i,
\end{equation*}
where each $t_i(x)$ ($i=0,\ldots,8$) is explicitly given in Appendix \ref{app:min-pol-Tdot}.

Next, we compute the discriminant of $P_{T^{\bullet}}$ with respect to $T^{\bullet}(x)$.
It has five factors and we can discard two of them for trivial reasons.
Another factor admits $0.0014891$ as positive root, which is smaller than $\rho_1 \approx 0.0367$ and can be discarded.
While another one has the positive root $0.53898$, larger than $1/4$ the dominant singularity of the generating function of irreducible quadrangulations.
It can then be discarded since the class of irreducible quadrangulations is contained in the class of simple 3-connected quadrangulations, which is in bijection with the class of 3-connected 4-regular planar graphs.
So the dominant singularity must be the smallest positive root $\tau\approx 0.24451$ of the remaining factor, namely
\begin{equation*}
	3645x^2 - 4752x + 944 = 0,
\end{equation*}
as claimed.
The Puiseux expansion of $T^{\bullet}(x)$ near $\tau$ is of the form
\begin{equation}\label{eq:puis_T}
	T^{\bullet}(x) = T^{\bullet}_0 + T^{\bullet}_2X^2 + T^{\bullet}_3X^3 + O(X^4),
\end{equation}
where $X = \sqrt{1 - {x}/{\tau}}$, $T^{\bullet}_0\approx 0.000057426$, $T^{\bullet}_2\approx -0.00092862$ and $T^{\bullet}_3\approx 0.0028525$.
We conclude by applying Lemma \ref{lem:analytic} to \eqref{eq:puis_T} and dividing the resulting estimate by $n$.

To prove the estimate on the number of 2-connected 4-regular planar graphs, we proceed in the same way and thus omit certain details that can be found in the \texttt{Maple} sessions.
Consider the system of equations \eqref{sys:Bdot} and eliminate all the  other variables to obtain a single irreducible bivariate polynomial equations in $x$ and $B^{\bullet}(x)$, as follows:
\begin{equation*}\label{poly:Bdot}
    P_{B^{\bullet}}(B^{\bullet}(x),x) = \sum_{i=0}^{29} b_i(x)\cdot B^{\bullet}(x)^i.
\end{equation*}
The discriminant of $P_{B^{\bullet}}$ with respect to $B^{\bullet}(x)$ has seven factors.
Only two of them have positive roots strictly smaller than one.
The smallest such root of one of the factors is $0.013756$, again smaller than $\rho_1$.
Thus, the dominant singularity $\beta$ of $B^{\bullet}(x)$ has to be the smallest positive root of the remaining factor, which is the one claimed.
The Puiseux expansion of $B^{\bullet}(x)$ near $\beta$ is of the form
\begin{equation}\label{eq:puis_B}
	B^{\bullet}(x) = B_0 + B_2X^2 + B_3X^3 + O(X^4),
\end{equation}
where $X = \sqrt{1 - {x}/{\beta}}$, $B_0\approx 0.000057583$, $B_2\approx -0.00098647$ and $B_3\approx 0.0032669$.
We conclude again by applying Lemma \ref{lem:analytic} to \eqref{eq:puis_B} and dividing the resulting estimate by $n$.

\paragraph{Proof of Theorem \ref{thm:maps}.}

The proof goes along the same lines as the two proofs above and we only briefly sketch it here.
We refer the reader to the accompanying \texttt{Maple} sessions.
First eliminate from \eqref{sys:M} and \eqref{sys:N} to obtain the minimal polynomials $P_M$ and $P_N$ satisfied by $M(x)$ and $N(x)$, respectively.
We compute the discriminants of $P_M$ with respect to $M(x)$ and of $P_N$ with respect to $N(x)$, and find in each case the unique factor with a positive root.
Then the smallest such root is the dominant singularity.
We conclude by applying Lemma \ref{lem:analytic} to the local expansions of $M(x)$ and $N(x)$ near their respective dominant singularities.

%
%
\section{Acknowledgements}

The authors are  grateful to Thibaut Verron for suggesting the use of evaluation and multivariate polynomial interpolation to obtain the minimal polynomials of $T_1(u,v)$ and $T_2(u,v)$.
Part of this work was done while the second author was a post-doctoral researcher under the direction of Manuel Kauers, at the Institute for Algebra of the Johannes Kepler Universit{\"a}t Linz, and Michael Drmota, at the Institute for Discrete Mathematics and Geometry of the Technische Universit{\"a}t Wien, and supported by the Special Research Program F50 \textit{Algorithmic and Enumerative Combinatorics} of the Austrian Science Fund. 
We are also grateful to the anonymous referees for their careful readings of the different versions of this paper, and whose comments and suggestions helped us improve the presentation.

This research was supported by grants MTM2017-82166-P, MDM-2014-0445, Beatriu de Pin\'os BP2020 funded by the H2020 COFUND project No 801370 and AGAUR (the Catalan agency
for management of university and research grants), PID2020-113082GB-I00, the Severo Ochoa and María de Maeztu Program for Centers and Units of Excellence in R\&D (CEX2020-001084-M), and the Marie Curie RISE research network 'RandNet' MSCA-RISE-2020-101007705.

%
%

   \bibliographystyle{abbrv}
\bibliography{biblio_4-regular}

%
%

\appendix

\section{Minimal polynomials for quadrangulations}\label{app:min-pol-B}

\textbf{Coeffients of $P_{B_0}(B_0(z,w),z,w) = \sum_{i=0}^2 p_{i,0}(z,w) B_0(z,w)^i$}:

\begin{itemize}
    \item
    $p_{2,0}(z,w) = 27z^2(wz - z - 1)^2(wz - z + 1)^3$.

    \item
    $p_{1,0}(z,w) =  - (wz - z - 1)(16w^8z^{12} - 144w^7z^{12} - 32w^7z^{11} + 560w^6z^{12} + 16w^7z^{10} + 224w^6z^{11} - 1232w^5z^{12} - 136w^6z^{10} - 672w^5z^{11} + 1680w^4z^{12} + 40w^6z^9 + 544w^5z^{10} + 1120w^4z^{11} - 1456w^3z^{12} - 296w^5z^9 - 1240w^4z^{10} - 1120w^3z^{11} + 784w^2z^{12} + 204w^5z^8 + 816w^4z^9 + 1680w^3z^{10} + 672w^2z^{11} - 240wz^{12} - 1052w^4z^8 - 1104w^3z^9 - 1336w^2z^{10} - 224wz^{11} + 32z^{12} - 4w^5z^6 + 368w^4z^7 + 2184w^3z^8 + 776w^2z^9 + 576wz^{10} + 32z^{11} - 36w^4z^6 - 1440w^3z^7 - 2280w^2z^8 - 264wz^9 - 104z^{10} - 10w^4z^5 + 444w^3z^6 + 2112w^2z^7 + 1196wz^8 + 32z^9 - w^4z^4 - 80w^3z^5 - 908w^2z^6 - 1376wz^7 - 252z^8 + 4w^3z^4 + 156w^2z^5 + 648wz^6 + 336z^7 + 2w^3z^3 - 10w^2z^4 - 104wz^5 - 144z^6 + 20w^2z^3 - 156wz^4 + 38z^5 + 42wz^3 + 163z^4 + 8wz^2 - 208z^3 - 2wz + 84z^2 - 18z + 1).$

    \item
    $p_{0,0}(z,w) = z(64w^{10}z^{14} - 704w^9z^{14} - 160w^9z^{13} + 3456w^8z^{14} + 32w^9z^{12} + 1472w^8z^{13} - 9984w^7z^{14} - 272w^8z^{12} - 6016w^7z^{13} + 18816w^6z^{14} + 224w^8z^{11} + 1264w^7z^{12} + 14336w^6z^{13} - 24192w^5z^{14} - 16w^8z^{10} - 2048w^7z^{11} - 3920w^6z^{12} - 21952w^5z^{13} + 21504w^4z^{14} + 408w^7z^{10} + 7696w^6z^{11} + 8176w^5z^{12} + 22400w^4z^{13} - 13056w^3z^{14} - 88w^7z^9 - 2536w^6z^{10} - 15712w^5z^{11} - 11312w^4z^{12} - 15232w^3z^{13} + 5184w^2z^{14} - 8w^7z^8 + 928w^6z^9 + 7800w^5z^{10} + 19120w^4z^{11} + 10192w^3z^{12} + 6656w^2z^{13} - 1216wz^{14} - 60w^6z^8 - 3304w^5z^9 - 14120w^4z^{10} - 14144w^3z^{11} - 5744w^2z^{12} - 1696wz^{13} + 128z^{14} + 8w^6z^7 + 664w^5z^8 + 5520w^4z^9 + 15816w^3z^{10} + 6128w^2z^{11} + 1840wz^{12} + 192z^{13} - 168w^5z^7 - 1676w^4z^8 - 4520w^3z^9 - 10808w^2z^{10} - 1376wz^{11} - 256z^{12} - 3w^5z^6 - 168w^4z^7 + 1624w^3z^8 + 1408w^2z^9 + 4136wz^{10} + 112z^{11} + 51w^4z^6 + 1712w^3z^7 - 388w^2z^8 + 232wz^9 - 680z^{10} + 33w^4z^5 - 514w^3z^6 - 2584w^2z^7 - 296wz^8 - 176z^9 + 2w^4z^4 + 244w^3z^5 + 1458w^2z^6 + 1496wz^7 + 140z^8 + 2w^3z^4 - 774w^2z^5 - 1563wz^6 - 296z^7 - 4w^3z^3 + 170w^2z^4 + 492wz^5 + 571z^6 - 54w^2z^3 - 98wz^4 + 5z^5 - 56wz^3 - 12z^4 - 7wz^2 + 162z^3 + 4wz - 77z^2 + 29z - 2).$
\end{itemize}
\textbf{Coeffients of $P_{B_1}(B_1(z,w),z,w) = \sum_{i=0}^2 p_{i,1}(z,w) B_1(z,w)^i$}:

\begin{itemize}
    \item
    $p_{2,1}(z,w) = 27z(wz - z - 1)^2(wz - z + 1)^3$.

    \item
    $p_{1,1}(z,w) = - 2w(wz - z - 1)(16w^7z^{11} - 112w^6z^{11} - 16w^6z^{10} + 336w^5z^{11} + 24w^6z^9 + 96w^5z^{10}
    - 560w^4z^{11} - 176w^5z^9 - 240w^4z^{10} + 560w^3z^{11} + 56w^5z^8 + 520w^4z^9 + 320w^3z^{10} - 336w^2z^{11}
    - 42w^5z^7 - 232w^4z^8 - 800w^3z^9 - 240w^2z^{10} + 112wz^{11} + 272w^4z^7 + 368w^3z^8 + 680w^2z^9 + 96wz^{10}
    - 16z^{11} - 126w^4z^6 - 684w^3z^7 - 272w^2z^8 - 304wz^9 - 16z^{10} + 2w^4z^5 + 548w^3z^6 + 840w^2z^7
    + 88wz^8 + 56z^9 - 146w^3z^5 - 888w^2z^6 - 506wz^7 - 8z^8 + 4w^3z^4 + 330w^2z^5 + 636wz^6 + 120z^7 - 36w^2z^4
    - 230wz^5 - 170z^6 + 2w^2z^3 + 24wz^4 + 44z^5 + w^2z^2 + 20wz^3 + 8z^4 + 10wz^2 - 76z^3 - 2wz + 79z^2 - 18z + 1).$

    \item
    $p_{0,1}(z,w) = - 4w^2z^2(16w^9z^{12} - 160w^8z^{12} - 16w^8z^{11} + 704w^7z^{12} + 24w^8z^{10} + 160w^7z^{11}
    - 1792w^6z^{12} - 232w^7z^{10} - 672w^6z^{11} + 2912w^5z^{12} + 56w^7z^9 + 968w^6z^{10} + 1568w^5z^{11}
    - 3136w^4z^{12} - 15w^7z^8 - 408w^6z^9 - 2280w^5z^{10} - 2240w^4z^{11} + 2240w^3z^{12} + 171w^6z^8
    + 1192w^5z^9 + 3320w^4z^{10} + 2016w^3z^{11} - 1024w^2z^{12} - 45w^6z^7 - 722w^5z^8 - 1800w^4z^9
    - 3064w^3z^{10} - 1120w^2z^{11} + 272wz^{12} + 2w^6z^6 + 330w^5z^7 + 1634w^4z^8 + 1480w^3z^9
    + 1752w^2z^{10} + 352wz^{11} - 32z^{12} - 67w^5z^6 - 850w^4z^7 - 2211w^3z^8 - 616w^2z^9
    - 568wz^{10} - 48z^{11} + 4w^5z^5 + 173w^4z^6 + 984w^3z^7 + 1799w^2z^8 + 88wz^9 + 80z^{10}
    - 19w^4z^5 - 40w^3z^6 - 477w^2z^7 - 812wz^8 + 8z^9 + 2w^4z^4 - 94w^3z^5 - 278w^2z^6
    + 22wz^7 + 156z^8 + w^4z^3 + 18w^3z^4 + 258w^2z^5 + 307wz^6 + 36z^7 + 9w^3z^3 - 106w^2z^4
    - 178wz^5 - 97z^6 - 2w^3z^2 + 55w^2z^3 + 195wz^4 + 29z^5 - 17w^2z^2 - 71wz^3 - 109z^4
    + w^2z + wz^2 - 21z^3 - wz + 44z^2 - 15z + 1).$
\end{itemize}

\newpage
\textbf{Coeffients of $P_{B_0^*}(B_0^*(z,w),z,w) = \sum_{i=0}^2 p_{i,2}(z,w) B_0^*(z,w)^i$}:

\begin{itemize}
    \item
    $p_{2,2}(z,w) = 27(wz - z + 1)^2$.

    \item
    $p_{1,2}(z,w) = 16w^5z^8 - 80w^4z^8 - 32w^4z^7 + 160w^3z^8 + 24w^4z^6 + 128w^3z^7 - 160w^2z^8 - 80w^3z^6 - 192w^2z^7 + 80wz^8 + 24w^3z^5 + 96w^2z^6 + 128wz^7 - 16z^8 + 12w^3z^4 - 72w^2z^5 - 48wz^6 - 32z^7 + 96w^2z^4 + 72wz^5 + 8z^6 + 12w^2z^3 - 228wz^4 - 24z^5 + 2w^2z^2 + 120wz^3 + 120z^4 + 8wz^2 - 132z^3 - 4wz + 98z^2 - 32z + 2.$

    \item
    $p_{0,2}(z,w) = 4z^3(8w^5z^7 - 40w^4z^7 - 16w^4z^6 + 80w^3z^7 + 8w^4z^5 + 64w^3z^6 - 80w^2z^7 - 24w^3z^5 - 96w^2z^6 + 40wz^7 + 20w^3z^4 + 24w^2z^5 + 64wz^6 - 8z^7 + 2w^3z^3 - 60w^2z^4 - 8wz^5 - 16z^6 + 29w^2z^3 + 60wz^4 - 2w^2z^2 - 64wz^3 - 20z^4 + 22wz^2 + 33z^3 + 2wz - 20z^2 + 25z - 2).$
\end{itemize}

\section{Minimal polynomials for 3-connected planar maps}\label{app:min-pol-T1T2}

\textbf{Coeffients of $P_{T_1}(T_1(u,v),u,v) = \sum_{i=0}^8 t_{i,1}(u,v) T_1(u,v)^i$}:

\begin{itemize}
    \item
    $t_{8,1}(u,v) = 1.$

    \item
    $t_{7,1}(u,v) = - 8u^2 + 16uv + 2u + 10.$

    \item
    $t_{6,1}(u,v) = 28u^4 - 112u^3v + 112u^2v^2 - 14u^3 + 28u^2v - 69u^2 + 152uv + 68u + 42.$

    \item
    $t_{5,1}(u,v) = - 56u^6 + 336u^5v - 672u^4v^2 + 448u^3v^3 + 42u^5 - 168u^4v + 168u^3v^2 + 204u^4 - 900u^3v + 984u^2v^2 - 408u^3 + 840u^2v - 92u^2 + 520uv + 352u + 96.$

    \item
    $t_{4,1}(u,v) = 70u^8 - 560u^7v + 1680u^6v^2 - 2240u^5v^3 + 1120u^4v^4 - 70u^7 + 420u^6v - 840u^5v^2 + 560u^4v^3 - 335u^6 + 2220u^5v - 4860u^4v^2 + 3520u^3v^3 + 1020u^5 - 4200u^4v + 4320u^3v^2 - 170u^4 - 988u^3v + 2728u^2v^2 - 1604u^3 + 3280u^2v + 394u^2 + 776uv + 692u + 129.$

    \item
    $t_{3,1}(u,v) = - 56u^{10} + 560u^9v - 2240u^8v^2 + 4480u^7v^3 - 4480u^6v^4 + 1792u^5v^5 + 70u^9 - 560u^8v + 1680u^7v^2 - 2240u^6v^3 + 1120u^5v^4 + 330u^8 - 2920u^7v + 9600u^6v^2 - 13920u^5v^3 + 7520u^4v^4 - 1360u^7 + 8400u^6v - 17280u^5v^2 + 11840u^4v^3 + 760u^6 - 1248u^5v - 4416u^4v^2 + 7744u^3v^3 + 2896u^5 - 11872u^4v + 12256u^3v^2 - 2482u^4 + 2756u^3v + 2736u^2v^2 - 1562u^3 + 4224u^2v + 1250u^2 + 424uv + 550u + 102.$

    \item
    $t_{2,1}(u,v) = 28u^{12} - 336u^{11}v + 1680u^{10}v^2 - 4480u^9v^3 + 6720u^8v^4 - 5376u^7v^5 + 1792u^6v^6 - 42u^{11} + 420u^{10}v - 1680u^9v^2 + 3360u^8v^3 - 3360u^7v^4 + 1344u^6v^5 - 195u^{10} + 2160u^9v - 9480u^8v^2 + 20640u^7v^3 - 22320u^6v^4 + 9600u^5v^5 + 1020u^9 - 8400u^8v + 25920u^7v^2 - 35520u^6v^3 + 18240u^5v^4 - 970u^8 + 4472u^7v - 3120u^6v^2 - 10144u^5v^3 + 12512u^4v^4 - 2584u^7 + 15936u^6v - 33024u^5v^2 + 22976u^4v^3 + 4122u^6 - 12600u^5v + 6024u^4v^2 + 5536u^3v^3 + 534u^5 - 6960u^4v + 9744u^3v^2 - 2829u^4 + 5636u^3v + 392u^2v^2 + 722u^3 + 1444u^2v + 991u^2 - 64uv + 104u + 44$.

    \item
    $t_{1,1}(u,v) =  - 8u^{14} + 112u^{13}v - 672u^{12}v^2 + 2240u^{11}v^3 - 4480u^{10}v^4 + 5376u^9v^5 - 3584u^8v^6 + 1024u^7v^7 + 14u^{13} - 168u^{12}v + 840u^{11}v^2 - 2240u^{10}v^3 + 3360u^9v^4 - 2688u^8v^5 + 896u^7v^6 + 64u^{12} - 852u^{11}v + 4680u^{10}v^2 - 13600u^9v^3 + 22080u^8v^4 - 19008u^7v^5 + 6784u^6v^6 - 408u^{11} + 4200u^{10}v - 17280u^9v^2 + 35520u^8v^3 - 36480u^7v^4 + 14976u^6v^5 + 548u^{10} - 3848u^9v + 8576u^8v^2 - 2944u^7v^3 - 11840u^6v^4 + 10880u^5v^5 + 1136u^9 - 9376u^8v + 29280u^7v^2 - 40960u^6v^3 + 21632u^5v^4 - 2854u^8 + 13828u^7v - 19608u^6v^2 + 3504u^5v^3 + 6464u^4v^4 + 850u^7 + 1248u^6v - 11112u^5v^2 + 10560u^4v^3 + 1392u^6 - 6412u^5v + 6344u^4v^2 + 736u^3v^3 - 1312u^5 + 2216u^4v + 1160u^3v^2 - 356u^4 + 2104u^3v - 504u^2v^2 + 626u^3 - 312u^2v - 90u^2 - 96uv - 40u + 8.$

    \item
    $t_{0,1}(u,v) = u^2(u^{14} - 16u^{13}v + 112u^{12}v^2 - 448u^{11}v^3 + 1120u^{10}v^4 - 1792u^9v^5 + 1792u^8v^6 - 1024u^7v^7 + 256u^6v^8 - 2u^{13} + 28u^{12}v - 168u^{11}v^2 + 560u^{10}v^3 - 1120u^9v^4 + 1344u^8v^5 - 896u^7v^6 + 256u^6v^7 - 9u^{12} + 140u^{11}v - 924u^{10}v^2 + 3360u^9v^3 - 7280u^8v^4 + 9408u^7v^5 - 6720u^6v^6 + 2048u^5v^7 + 68u^{11} - 840u^{10}v + 4320u^9v^2 - 11840u^8v^3 + 18240u^7v^4 - 14976u^6v^5 + 5120u^5v^6 - 118u^{10} + 1092u^9v - 3768u^8v^2 + 5344u^7v^3 - 672u^6v^4 - 5568u^5v^5 + 3968u^4v^6 - 196u^9 + 2032u^8v - 8512u^7v^2 + 17984u^6v^3 - 19136u^5v^4 + 8192u^4v^5 + 724u^8 - 4760u^7v + 10848u^6v^2 - 8608u^5v^3 - 1344u^4v^4 + 3328u^3v^5 - 514u^7 + 1488u^6v + 1368u^5v^2 - 7008u^4v^3 + 4864u^3v^4 + 58u^6 + 352u^5v - 1252u^4v^2 + 80u^3v^3 + 1232u^2v^4 + 40u^5 - 96u^4v - 392u^3v^2 + 1072u^2v^3 - 8u^4 + 292u^2v^2 + 128uv^3 + 32uv^2 - 16v^2).$
\end{itemize}

\newpage
\textbf{Coeffients of $P_{T_2}(T_2(u,v),u,v) = \sum_{i=0}^8 t_{i,2}(u,v) T_2(u,v)^i$}:

\begin{itemize}
    \item
    $t_{8,2}(u,v) = (v + 1)^5.$

    \item
    $t_{7,2}(u,v) = 2v(v + 1)^4(8uv + 8u - 4v + 1).$

    \item
    $t_{6,2}(u,v) = 8v^2(v + 1)^3(14u^2v^2 + 28u^2v - 14uv^2 + 14u^2 - 9uv + 2v^2 + 5u - 5v - 2).$

    \item
    $t_{5,2}(u,v) = 8v^3(v + 1)^2(56u^3v^3 + 168u^3v^2 - 84u^2v^3 + 168u^3v - 129u^2v^2 + 24uv^3 + 56u^3 - 6u^2v - 48uv^2 + 39u^2 - 96uv + 16v^2 - 24u - 6).$

    \item
    $t_{4,2}(u,v) = 16v^4(v + 1)(70u^4v^4 + 280u^4v^3 - 140u^3v^4 + 420u^4v^2 - 340u^3v^3 + 60u^2v^4 + 280u^4v - 180u^3v^2 - 90u^2v^3 + 70u^4 + 100u^3v - 417u^2v^2 + 92uv^3 + 80u^3 - 324u^2v + 94uv^2 - 57u^2 - 20uv + 25v^2 - 22u + 22v + 2).$

    \item
    $t_{3,2}(u,v) = 32v^5(56u^5v^5 + 280u^5v^4 - 140u^4v^5 + 560u^5v^3 - 465u^4v^4 + 80u^3v^5 + 560u^5v^2 - 460u^4v^3 - 80u^3v^4 + 280u^5v + 10u^4v^2 - 788u^3v^3 + 208u^2v^4 + 56u^5 + 240u^4v - 1084u^3v^2 + 412u^2v^3 + 95u^4 - 524u^3v + 169u^2v^2 + 96uv^3 - 68u^3 - 66u^2v + 182uv^2 - 31u^2 + 108uv + 19v^2 + 22u + 26v + 8).$

    \item
    $t_{2,2}(u,v) = 64v^6(28u^6v^5 + 140u^6v^4 - 84u^5v^5 + 280u^6v^3 - 270u^5v^4 + 60u^4v^5 + 280u^6v^2 - 240u^5v^3 - 90u^4v^4 + 140u^6v + 60u^5v^2 - 672u^4v^3 + 232u^3v^4 + 28u^6 + 180u^5v - 876u^4v^2 + 440u^3v^3 + 66u^5 - 396u^4v + 164u^3v^2 + 150u^2v^3 - 42u^4 - 64u^3v + 265u^2v^2 - 20u^3 + 155u^2v + 12uv^2 + 40u^2 + 18uv + 12u + 7v + 5)$.

    \item
    $t_{1,2}(u,v) = 128v^7(8u^7v^5 + 40u^7v^4 - 28u^6v^5 + 80u^7v^3 - 87u^6v^4 + 24u^5v^5 + 80u^7v^2 - 68u^6v^3 - 48u^5v^4 + 40u^7v + 38u^6v^2 - 300u^5v^3 + 128u^4v^4 + 8u^7 + 72u^6v - 372u^5v^2 + 228u^4v^3 + 25u^6 - 156u^5v + 67u^4v^2 + 112u^3v^3 - 12u^5 - 38u^4v + 168u^3v^2 - 5u^4 + 76u^3v + u^2v^2 + 20u^3 - 8u^2v + 3u^2 - 16uv - 6u + 1).$

    \item
    $t_{0,2}(u,v) = 256u^2v^8(u^6v^5 + 5u^6v^4 - 4u^5v^5 + 10u^6v^3 - 12u^5v^4 + 4u^4v^5 + 10u^6v^2 - 8u^5v^3 - 10u^4v^4 + 5u^6v + 8u^5v^2 - 55u^4v^3 + 28u^3v^4 + u^6 + 12u^5v - 65u^4v^2 + 46u^3v^3 + 4u^5 - 25u^4v + 8u^3v^2 + 33u^2v^3 - u^4 - 10u^3v + 38u^2v^2 + 5u^2v + 12uv^2 + 4uv - v).$
\end{itemize}

\section{Minimal polynomials for 3-connected planar graphs}\label{app:min-pol-Tdot}

\textbf{Coeffients of $P_{T^{\bullet}}(T^{\bullet}(x),x) = \sum_{i=0}^8 t_i(x) T^{\bullet}(x)^i$}:

\begin{itemize}
    \item
    $t_8(x) = 16777216.$

    \item
    $t_7(x) = 4194304(x + 1)(5 - 4x).$

    \item
    $t_6(x) = 7340032x^4 - 3670016x^3 - 18087936x^2 + 17825792x + 11010048.$

    \item
    $t_5(x) = -1835008x^6 + 1376256x^5 + 6684672x^4 - 13369344x^3 - 3014656x^2 + 11534336x + 3145728.$

    \item
    $t_4(x) = 286720x^8 - 286720x^7 - 1372160x^6 + 4177920x^5 - 696320x^4 - 6569984x^3 + 1613824x^2 + 2834432x + 528384.$

    \item
    $t_3(x) = -28672x^{10} + 35840x^9 + 168960x^8 - 696320x^7 + 389120x^6 + 1482752x^5 - 1270784x^4 - 799744x^3 + 640000x^2 + 281600x + 52224.$

    \item
    $t_2(x) = 1792x^{12} - 2688x^{11} - 12480x^{10} + 65280x^9 - 62080x^8 - 165376x^7 + 263808x^6 + 34176x^5 - 181056x^4 + 46208x^3 + 63424x^2 + 6656x + 2816.$

    \item
    $t_1(x) = -64x^{14} + 112x^{13} + 512x^{12} - 3264x^{11} + 4384x^{10} + 9088x^9 - 22832x^8 + 6800x^7 + 11136x^6- 10496x^5 - 2848x^4 + 5008x^3 - 720x^2 - 320x + 64.$

    \item
    $t_0(x) = x^6(x^2 + 4x - 1)(x^8 - 6x^7 + 16x^6 - 2x^5 - 94x^4 + 178x^3 - 82x^2 - 8x + 8).$
\end{itemize}

\section{Minimal polynomials for simple maps}\label{app:min-pol-maps}

\textbf{Coeffients of $P_{M}(M(x),x) = \sum_{i=0}^4 m_i(x) M(x)^i$}:

\begin{itemize}
    \item
    $m_4(x) = (2x^2 + 3x + 3)^6(x + 1)^4(2x + 1)^2(x - 1)^2x^6.$

    \item
    $m_3(x) =  - 2x^4(x - 1)(64x^{14} + 480x^{13} + 1440x^{12} + 2496x^{11} + 276x^{10} - 11546x^9 - 26420x^8 - 19509x^7 + 6393x^6 + 19014x^5 + 12975x^4 + 4608x^3 + 702x^2 - 135x - 54)(x + 1)^3(2x^2 + 3x + 3)^3.$

    \item
    $m_2(x) =  - x^2(4608x^{26} + 39936x^{25} + 178688x^{24} + 520704x^{23} + 1094336x^{22} + 1543680x^{21} + 245408x^{20} - 8566240x^{19} - 32715326x^{18} - 59854300x^{17} - 53501976x^{16} - 7389020x^{15} + 36335841x^{14} + 52316608x^{13} + 51994151x^{12} + 41986758x^{11} + 22019337x^{10} + 1419738x^9 - 10788681x^8 - 14542164x^7 - 13339809x^6 - 9695160x^5 - 5505759x^4 - 2350134x^3 - 725355x^2 - 148230x - 14823)(x + 1)^2.$

    \item
    $m_1(x) = (x + 1)(27648x^{31} + 322560x^{30} + 1847808x^{29} + 6918144x^{28} + 18257536x^{27} + 34084608x^{26} + 43349120x^{25} + 40257248x^{24} + 50673996x^{23} + 88945348x^{22} + 99083870x^{21} + 51494754x^{20} + 9311464x^{19} - 4802994x^{18} - 30341844x^{17} - 55300330x^{16} - 51227730x^{15} - 36294430x^{14} - 27913810x^{13} - 19958630x^{12} - 10165107x^{11} - 3664163x^{10} - 1317483x^9 - 472215x^8 + 95764x^7 + 263384x^6 + 140418x^5 + 22266x^4 - 10179x^3 - 6345x^2 - 783x + 459).$

    \item
    $m_0(x) =  - x^6(x^2 + 4x - 1)(34560x^{26} + 331776x^{25} + 1792512x^{24} + 6458368x^{23} + 16625952x^{22} + 29597056x^{21} + 34923536x^{20} + 27157632x^{19} + 14306863x^{18} + 2833960x^{17} - 8516393x^{16} - 17003008x^{15} - 18205069x^{14} - 14628522x^{13} - 10556741x^{12} - 6840238x^{11} - 3542614x^{10} - 1345848x^9 - 348274x^8 - 14552x^7 + 80947x^6 + 78354x^5 + 38619x^4 + 8694x^3 - 1215x^2 - 1512x - 459).$
\end{itemize}
\textbf{Coeffients of $P_{N}(N(x),x) = \sum_{i=0}^4 n_i(x) N(x)^i$}:

\begin{itemize}
    \item
    $n_4(x) = (x^2 + x + 2)^6(x + 1)^4(x - 1)^2x^3.$

    \item
    $n_3(x) =  - 2x^2(x - 1)(2x^{11} + 10x^{10} + 21x^9 + 39x^8 - 70x^7 - 282x^6 - 329x^5 + 112x^4 + 271x^3 + 67x^2 + 25x + 6)(x + 1)^3(x^2 + x + 2)^3.$

    \item
    $n_2(x) =   - x(18x^{22} + 84x^{21} + 344x^{20} + 832x^{19} + 1881x^{18} + 2362x^{17} - 213x^{16} - 22272x^{15} - 59887x^{14} - 60780x^{13} - 48821x^{12} + 2482x^{11} + 70283x^{10} + 76870x^9 + 64053x^8 + 35032x^7 - 9763x^6 - 33312x^5 - 24499x^4 - 8394x^3 - 2308x^2 - 328x - 48)(x + 1)^2.$

    \item
    $n_1(x) = (x + 1)(108x^{26} + 828x^{25} + 3798x^{24} + 12498x^{23} + 27832x^{22} + 44416x^{21} + 45122x^{20} + 52780x^{19} + 83480x^{18} + 38444x^{17} + 52052x^{16} + 63060x^{15} - 24937x^{14} + 18259x^{13} - 10897x^{12} - 78617x^{11} - 16174x^{10} - 16450x^9 - 26700x^8 - 984x^7 + 3531x^6 + 3159x^5 + 361x^4 - 641x^3 + 144x^2 - 48x + 8).$

    \item
    $n_0(x) =   - x^6(x^2 + 4x - 1)(135x^{21} + 756x^{20} + 3843x^{19} + 10936x^{18} + 25151x^{17} + 30928x^{16} + 31103x^{15} + 27536x^{14} + 3310x^{13} - 7148x^{12} - 10618x^{11} - 21620x^{10} - 13409x^9 - 5500x^8 - 4305x^7 - 60x^6 + 965x^5 + 604x^4 + 201x^3 - 104x^2 + 8x - 8).$
\end{itemize}

\end{document}